\documentclass{amsart}
 \usepackage{amssymb, latexsym}

\DeclareMathOperator{\ar}{\mathrm{Area}}
\DeclareMathOperator{\len}{\mathrm{Length}}

\begin{document}
 \bibliographystyle{plain}

 \newtheorem*{definition}{Definition}
 \newtheorem{theorem}{Theorem}
 \newtheorem{lemma}{Lemma}
 \newtheorem{corollary}{Corollary}
 \newcommand{\mc}{\mathcal}
 \newcommand{\F}{\mc{F}}
 \newcommand{\T}{\mc{T}}
 \newcommand{\lf}{\left\lfloor}
 \newcommand{\rf}{\right\rfloor}
 \newcommand{\rar}{\rightarrow}
 \newcommand{\mbb}{\mathbb}
 \newcommand{\R}{\mbb{R}}
 \newcommand{\N}{\mbb{N}}
 \newcommand{\Q}{\mbb{Q}}
 \newcommand{\Z}{\mbb{Z}}
 \newcommand{\Zv}{\Z^2_{\text{vis}}}
\title[Numerators of differences of nonconsecutive Farey fractions]{Numerators of differences of\\ nonconsecutive Farey fractions}
\author{Alan K. Haynes}
\thanks{Research supported by EPSRC grant EP/F027028/1}
\address{Department of Mathematics, University of York, Heslington, York YO10 5DD, UK}
\email{akh502@york.ac.uk}
 \allowdisplaybreaks


\begin{abstract}
An elementary but useful fact is that the numerator of the difference of two consecutive Farey fractions is equal to one. For triples of consecutive fractions the numerators of the differences are well understood and have applications to several interesting problems. In this paper we investigate numerators of differences of fractions which are farther apart. We establish algebraic identities between such differences which then allow us to calculate their average values by using properties of a measure preserving transformation of the Farey triangle.
\end{abstract}


\maketitle

\section{Introduction and statement of results}

For $Q\in\mathbb{N}$ the Farey fractions of order $Q$ are defined as \[\F_Q=\left\{\frac{a}{q}\in\Q : 1\leq q\leq Q, 0< a\le q,  (a,q)=1\right\},\] where $(a,q)$ denotes the greatest common divisor of $a$ and $q$. Let us write \[\F_Q=\{\gamma_1,\gamma_2,\ldots ,\gamma_{N(Q)}\}\] so that
$1/Q=\gamma_1<\gamma_2<\cdots <\gamma_{N(Q)}=1$ and extend this sequence by requiring that $\gamma_{i+N(Q)}=\gamma_i+1$ for
all $i\in\Z$. For each $i$ we will write $\gamma_i=p_i/q_i$ with
$(p_i,q_i)=1$.

A fundamental property of the Farey fractions which makes them a natural choice for many problems is that for any $i$, \[p_iq_{i-1}-p_{i-1}q_i=1.\] This follows from the fact that any parallelogram in the plane whose vertices are integer lattice points and whose closure contains no other lattice points has area one (see \cite{hardy1979}). Numerators of differences of nonconsecutive Farey fractions are also important in many problems. Since they are the object of study in this paper we make the following definition.
\begin{definition}
Given a positive integer $k$ and a fraction $\gamma_i$ in $\F_Q$
we define the $k-$index of $\gamma_i$ by
\begin{equation*}
\nu_k(\gamma_i)=p_{i+k-1}q_{i-1}-p_{i-1}q_{i+k-1}.
\end{equation*}
\end{definition}
Note that this definition depends on $Q$. The functions $\nu_k$ are easily
seen to be periodic in the sense that
\begin{align*}
\nu_k(\gamma_{i+N(Q)})&=\nu_k(\gamma_i).
\end{align*}
We have chosen the terminology {\em $k-$index} because the quantity $\nu_2(\gamma_i)$ has already been named the {\em index} of the fraction $\gamma_i$ in $\F_Q$ \cite{hall2003}. This index has been extensively studied and has applications to some interesting problems in number theory. To give the reader some feel for how it behaves, it is easy to show that as $i$ and $Q$ vary, $\nu_2(\gamma_i)$ takes all positive integer values. Furthermore the frequencies with which the different values occur can be estimated asymptotically as a function of $Q$. We will give more details about $\nu_2$ at the end of this section.

Numerators of differences of Farey fractions which are farther apart, which correspond to the values of $\nu_k(\gamma_i)$ for $k>2$, have not yet been investigated. However they arise naturally in problems where the denominators of the fractions are restricted to arithmetic progressions with composite moduli. Such problems have recently been considered in \cite{alkan2006a}, \cite{alkan2006b}, and \cite{cobeli2006}, and they have potential applications to billiards in which the source is centered not at the origin but at a point with nonzero rational coordinates. There is also a connection between such sequences of fractions and the Riemann Hypothesis for Dirichlet L-functions, but we will not say anything more about this now.

The first main result of this paper is the following theorem, which establishes a relatively simple algebraic relationship between the values taken on by $\nu_k$ and those taken on by $\nu_2$. The quantities $K_{k-1}$ which appear in the statement of the theorem are the convergent polynomials, and they are defined below.
\begin{theorem}\label{algidthm}
Let $Q$ be fixed and let $k$ be any positive integer. Then for each
fraction $\gamma_i$ in $\F_Q$ we have that
\begin{equation}\label{contiden}
\nu_k(\gamma_i)=\left(\frac{2k-1}{2}\right)K_{k-1}\left(-\nu_2(\gamma_i),\nu_2(\gamma_{i+1}),\ldots
,(-1)^{k-1}\nu_2(\gamma_{i+k-2})\right),
\end{equation}
where $(\frac{\cdot}{2})$ denotes the Kronecker symbol, defined by
\begin{equation*}
\left(\frac{n}{2}\right)=\begin{cases}0&\text{if }2|n,\\1&\text{if }
n\equiv\pm 1\mod 8,\\-1&\text{if }n\equiv\pm 3\mod
8.\end{cases}
\end{equation*}
\end{theorem}
Our second main result is the next theorem, which gives us an asymptotic estimate for the average value of $\nu_k$ as $Q\rar\infty$.
\begin{theorem}\label{asympthm}
For each integer $k\ge 0$ there exists a real constant $B(k)$
for which
\begin{equation}
\frac{1}{N(Q)}\sum_{i=1}^{N(Q)}\nu_k(\gamma_i)=B(k)+O_k\left(\frac{(\log Q)^2}{Q}\right)
\end{equation}
as $Q\rar\infty$. Furthermore for each $k$ the constant $B(k)$ can be computed by using the formula (\ref{Bkval}).
\end{theorem}
The convergent polynomials which appear in Theorem \ref{algidthm} are
defined by
\begin{align*}
K_0(\cdot )=1\quad\text{and}\quad K_1(x_1)=x_1,
\end{align*}
and then recursively by
\begin{align}
 K_n(x_1,x_2,\ldots,x_n)=&x_nK_{n-1}(x_1,x_2,\ldots
,x_{n-1})+K_{n-2}(x_1,x_2,\ldots ,x_{n-2}).\label{contrecur}
\end{align}
Thus each $K_n$ is an element of $\Z [x_1,x_2,\ldots
,x_n]$ and is linear in each of the variables $x_1,x_2,\ldots ,x_n$.
For example we have
\begin{align}
K_2(x_1,x_2)=&x_1x_2+1\label{K2form}\\
\intertext{and}
K_3(x_1,x_2,x_3)=&x_1x_2x_3+x_1+x_3.\label{K3form}
\end{align}
Some basic facts about these polynomials can be found in \cite{graham1994}.

Now for later reference let us say a few more words about $\nu_2$. The index of the fraction $\gamma_i$ was defined in \cite{hall2003} as
\begin{equation}\label{nudef1}
\nu(\gamma_i)=\frac{q_{i-1}+q_{i+1}}{q_i}.
\end{equation}
By the basic properties of Farey fractions we have that
\begin{align}
\nu(\gamma_i)=\frac{q_{i-1}+q_{i+1}}{q_i}&=q_{i+1}q_{i-1}\left(\frac{1}{q_{i-1}q_i}+\frac{1}{q_iq_{i+1}}\right)\nonumber\\
&=q_{i-1}q_{i+1}(\gamma_{i+1}-\gamma_i+\gamma_i-\gamma_{i-1})\label{nudef3}\\
&=p_{i+1}q_{i-1}-p_{i-1}q_{i+1},\nonumber
\end{align}
which shows that $\nu(\gamma_i)$ and $\nu_2(\gamma_i)$ are indeed equal. As recorded in \cite[(1.4)]{hall2003} another formula for the index is given by
\begin{equation}\label{indform1}
\nu_2(\gamma_i)=\lf\frac{Q+q_{i-1}}{q_i}\rf.
\end{equation}
Hall and Shiu \cite{hall2003} proved that
\begin{equation}\label{indsum1}
\sum_{i=1}^{N(Q)}\nu(\gamma_i)=3N(Q)-1.
\end{equation}
In the same paper they also proved asymptotic formulas for the sum
of the square moments of $\nu$, as well as for several other related
quantities. Boca, Gologan, and Zaharescu \cite{boca2002} extended
this result by finding asymptotic formulas for all moments of the
index which lie in $(0,2)$. Among other things they also proved that
for $h\ge 1$ there exists a constant $A(h)$ for which
\begin{equation}\label{twistsum}
\sum_{i=1}^{N(Q)}\nu(\gamma_i)\nu(\gamma_{i+h})=A(h)N(Q)+O_h(Q(\log Q)^2).
\end{equation}
It was shown in \cite{boca2002} that the constant $A(h)$ is $O(\log h)$, and it was also conjectured that this could be reduced to $O(1)$. We will say more about this and about bounds for our constants $B(k)$ in Section 4.


\section{Proof of Theorem \ref{algidthm}}

We begin by proving the following identity.
\begin{lemma}\label{detlemma}
Choose $Q\ge 2$ and $\gamma_i\in\F_Q$. Then for $k\ge 3$ we have
\begin{equation*}
\begin{pmatrix}\nu_{k-1}(\gamma_i)&\nu_k(\gamma_i)\\\nu_{k-2}(\gamma_{i+1})&\nu_{k-1}(\gamma_{i+1})\end{pmatrix}\in
SL_2(\Z).
\end{equation*}
\end{lemma}
\begin{proof}
First of all we have that
\begin{align}
\nu_{k-1}(\gamma_i)\nu_{k-1}(\gamma_{i+1})=&(p_{i+k-2}q_{i-1}-p_{i-1}q_{i+k-2})(p_{i+k-1}q_i-p_iq_{i+k-1})\nonumber\\
=&p_{i+k-2}p_{i+k-1}q_{i-1}q_i+p_{i-1}p_iq_{i+k-2}q_{i+k-1}\label{id1eq1}\\
&-p_{i-1}p_{i+k-1}q_iq_{i+k-2}-p_ip_{i+k-2}q_{i-1}q_{i+k-1}.\nonumber
\end{align}
The sum of the first two terms on the right hand side of this
equation is
\begin{align}
p_{i+k-2}p_{i+k-1}q_{i-1}q_i&+p_{i-1}p_iq_{i+k-2}q_{i+k-1}\nonumber\\
=&p_{i+k-1}q_{i-1}(\nu_{k-2}(\gamma_{i+1})+p_iq_{i+k-2})\nonumber\\
&+p_{i-1}q_{i+k-1}(p_{i+k-2}q_i-\nu_{k-2}(\gamma_{i+1}))\nonumber\\
=&\nu_k(\gamma_i)\nu_{k-2}(\gamma_{i+1})\label{id1eq2}\\
&+p_ip_{i+k-1}q_{i-1}q_{i+k-2}+p_{i-1}p_{i+k-2}q_iq_{i+k-1}.\nonumber
\end{align}
By the determinant property of Farey fractions we also have that
\begin{align}
p_ip_{i+k-1}q_{i-1}q_{i+k-2}&+p_{i-1}p_{i+k-2}q_iq_{i+k-1}\nonumber\\
&-p_{i-1}p_{i+k-1}q_iq_{i+k-2}-p_ip_{i+k-2}q_{i-1}q_{i+k-1}\nonumber\\
=&(p_{i+k-1}q_{i+k-2}-p_{i+k-2}q_{i+k-1})(p_iq_{i-1}-p_{i-1}q_i)=1\label{id1eq3}
\end{align}
Combining (\ref{id1eq1}), (\ref{id1eq2}), and (\ref{id1eq3}) we find
that
\begin{equation}\label{id1eq4}
\nu_{k-1}(\gamma_i)\nu_{k-1}(\gamma_{i+1})-\nu_k(\gamma_i)\nu_{k-2}(\gamma_{i+1})=1,
\end{equation}
and this proves the lemma.
\end{proof}
Solving for $\nu_k(\gamma_i)$ in (\ref{id1eq4}) gives us the formula
\begin{equation}\label{id1eq5}
\nu_k(\gamma_i)=\frac{\nu_{k-1}(\gamma_i)\nu_{k-1}(\gamma_{i+1})-1}{\nu_{k-2}(\gamma_{i+1})}
\end{equation}
Induction now shows that $\nu_k(\gamma_i)$ is given by a rational function
evaluated at the integers $\{\nu_2(\gamma_j)\}_{j=i}^{i+k-2}$. To prove that this rational function is actually
a polynomial we will use the following
result.
\begin{lemma}\label{identity2}
For $k\ge 3$ and $\gamma_i\in\F_Q$ we have
\begin{equation}\label{id2eq1}
\nu_k(\gamma_i)=\nu_2(\gamma_{i+k-2})\nu_{k-1}(\gamma_i)-\nu_{k-2}(\gamma_i).
\end{equation}
\end{lemma}
\begin{proof}
Our proof is by induction on $k$. Setting $k=3$ in identity (\ref{id1eq5})
and using the fact that $\nu_1\equiv 1$ gives us
\begin{align*}
\nu_3(\gamma_i)=\frac{\nu_2(\gamma_i)\nu_2(\gamma_{i+1})-1}{\nu_1(\gamma_{i+1})}=\nu_2(\gamma_i)\nu_2(\gamma_{i+1})-\nu_1(\gamma_i).
\end{align*}
Now assume the truth of our identity for all integers $3\le j\le
k-1$. Then using (\ref{id1eq5}) together with the inductive
hypothesis we find that
\begin{align}
\nu_k(\gamma_i)+&\nu_{k-2}(\gamma_i)=\frac{\nu_{k-2}(\gamma_i)\nu_{k-2}(\gamma_{i+1})+\nu_{k-1}(\gamma_i)\nu_{k-1}(\gamma_{i+1})-1}{\nu_{k-2}(\gamma_{i+1})}\nonumber\\
=&\frac{\nu_{k-2}(\gamma_i)\nu_{k-2}(\gamma_{i+1})+\nu_{k-1}(\gamma_i)(\nu_2(\gamma_{i+k-2})\nu_{k-2}(\gamma_{i+1})-\nu_{k-3}(\gamma_{i+1}))-1}{\nu_{k-2}(\gamma_{i+1})}\nonumber\\
=&\nu_2(\gamma_{i+k-2})\nu_{k-1}(\gamma_i)+\frac{\nu_{k-2}(\gamma_i)\nu_{k-2}(\gamma_{i+1})-\nu_{k-1}(\gamma_i)\nu_{k-3}(\gamma_{i+1})-1}{\nu_{k-2}(\gamma_{i+1})}.\label{id2eq2}
\end{align}
Now substituting $k-1$ for $k$ in (\ref{id1eq4}) we find that the
numerator of the fraction in (\ref{id2eq2}) is equal to zero, and
this finishes the proof.
\end{proof}
With Lemma \ref{identity2} in hand we are ready to prove Theorem \ref{algidthm}.
\begin{proof}[Proof of Theorem \ref{algidthm}]
Again the proof is by induction on $k$. For $k=1$ the function on
the right hand side of (\ref{contiden}) is
\begin{equation*}
\left(\frac{1}{2}\right)K_0(\cdot)=1=\nu_1(\gamma_i),
\end{equation*}
and for $k=2$ it is
\begin{equation*}
\left(\frac{3}{2}\right)K_1(-\nu_2(\gamma_i))=\nu_2(\gamma_i).
\end{equation*}
Now assume that $k\ge 3$ and that our result is true for all
integers $1\le j\le k-1$. Then using (\ref{id2eq1}) gives us
\begin{align}
\nu_k(\gamma_i)=&\nu_2(\gamma_{i+k-2})\nu_{k-1}(\gamma_i)-\nu_{k-2}(\gamma_i)\nonumber\\
=&\left(\frac{2k-3}{2}\right)\nu_2(\gamma_{i+k-2})K_{k-2}\left(-\nu_2(\gamma_i),\ldots
,(-1)^{k-2}\nu_2(\gamma_{i+k-3})\right)\label{id3eq1}\\
&-\left(\frac{2k-5}{2}\right)K_{k-3}\left(-\nu_2(\gamma_i),\ldots
,(-1)^{k-3}\nu_2(\gamma_{i+k-4})\right).\nonumber
\end{align}
We also observe by (\ref{contrecur}) that
\begin{align}
&\left(\frac{2k-1}{2}\right)K_{k-1}\left(-\nu_2(\gamma_i),\ldots
,(-1)^{k-1}\nu_2(\gamma_{i+k-2})\right)\nonumber\\
=&\left(\frac{2k-1}{2}\right)\left((-1)^{k-1}\nu_2(\gamma_{i+k-2})K_{k-2}\left(-\nu_2(\gamma_i),\ldots
,(-1)^{k-2}\nu_2(\gamma_{i+k-3})\right)\right.\nonumber\\
&\left.\qquad\qquad\qquad+K_{k-3}\left(-\nu_2(\gamma_i),\ldots
,(-1)^{k-3}\nu_2(\gamma_{i+k-4})\right)\right)\nonumber\\
=&\left(\frac{2k-1}{2}\right)\left(\frac{4k-3}{2}\right)\nu_2(\gamma_{i+k-2})K_{k-2}\left(-\nu_2(\gamma_i),\ldots
,(-1)^{k-2}\nu_2(\gamma_{i+k-3})\right)\nonumber\\
&\qquad\qquad\qquad+\left(\frac{2k-1}{2}\right)K_{k-3}\left(-\nu_2(\gamma_i),\ldots
,(-1)^{k-3}\nu_2(\gamma_{i+k-4})\right)\nonumber\\
=&\left(\frac{2k-3}{2}\right)\nu_2(\gamma_{i+k-2})K_{k-2}\left(-\nu_2(\gamma_i),\ldots
,(-1)^{k-2}\nu_2(\gamma_{i+k-3})\right)\label{id3eq2}\\
&\qquad\qquad\qquad-\left(\frac{2k-5}{2}\right)K_{k-3}\left(-\nu_2(\gamma_i),\ldots
,(-1)^{k-3}\nu_2(\gamma_{i+k-4})\right),\nonumber
\end{align}
Note that here we have used the properties of the Kronecker symbol to write
\begin{align*}
\left(\frac{2k-1}{2}\right)\left(\frac{4k-3}{2}\right)=\left(\frac{8k^2-10k+3}{2}\right)=\left(\frac{2k-3}{2}\right)
\end{align*}
and
\begin{align*}
\left(\frac{2k-1}{2}\right)=-\left(\frac{2k-5}{2}\right).
\end{align*}
Finally, combining (\ref{id3eq1}) and (\ref{id3eq2}) establishes
(\ref{contiden}).
\end{proof}
To conclude this section we would like to point out that Theorem \ref{algidthm} together with (\ref{twistsum}) immediately yields a proof of Theorem \ref{asympthm} in the case when $k=3$. Indeed using (\ref{K2form}) and (\ref{twistsum}) we find that
\begin{align*}
\sum_{i=1}^{N(Q)}\nu_3(\gamma_i)&=\sum_{i=1}^{N(Q)}\left(\frac{5}{2}\right)K_2(-\nu_2(\gamma_i),\nu_2(\gamma_{i+1}))\\
&=\sum_{i=1}^{N(Q)}(\nu_2(\gamma_i)\nu_2(\gamma_{i+1})-1)\\
&=(A(1)-1)N(Q)+O(Q(\log Q)^2),
\end{align*}
and combining this with the observation that
\[N(Q)=\frac{3Q^2}{\pi^2}+O(Q(\log Q)^2)\]
finishes the argument.


\section{Proof of Theorem \ref{asympthm}}
Now we will show how the results of Section 2 can be used to prove Theorem \ref{asympthm}. The notation and many of the ideas in this section closely follow \cite{boca2001} and the proof of \cite[Theorem 1.1]{boca2002}.

We begin by defining the Farey triangle $\T\subseteq[0,1]^2$ by
\begin{equation*}
\T=\{(x,y)\in [0,1]^2:x+y>1\},
\end{equation*}
and we define the map $T:[0,1]^2\rar [0,1]^2$ by
\begin{equation*}
T(x,y)=\left(y,\left[\frac{1+x}{y}\right]y-x\right),
\end{equation*}
where $[x]$ denotes the greatest integer less than or equal to $x$. As shown in \cite{boca2001}, the map $T$ is a one-to-one area preserving transformation of $\T$ onto itself.
Now for each positive integer $k$ let
\begin{equation*}
\T_k=\left\{(x,y)\in\T:\left[\frac{1+x}{y}\right]=k\right\}.
\end{equation*}
Then the set $\T$ is the disjoint union of the sets $\T_k$ and we also have that
\begin{equation*}
T(x,y)=(y,ky-x)\quad\text{ for all }\quad(x,y)\in\T_k.
\end{equation*}
Also of importance to us are the maps $\kappa_i:\T\rar\Z^+$ defined for positive integers $i$ by
\begin{align*}
\kappa_1&=\left[\frac{1+x}{y}\right]\quad\text{ and}\\
\kappa_{i+1}&=\kappa_i\circ T(x,y)=\kappa_1\circ T^i(x,y).
\end{align*}
It is clear from the definition that
\begin{equation}\label{Tkform1}
\T_k=\{(x,y)\in\T:\kappa_1(x,y)=k\}.
\end{equation}
One useful property of the function $T$ is the fact
that for any integer $j$ we have
\begin{equation*}
T\left(\frac{q_{j-1}}{Q},\frac{q_j}{Q}\right)=\left(\frac{q_j}{Q},\frac{q_{j+1}}{Q}\right).
\end{equation*}
Using this fact we find that for any non-negative integer $i$
\begin{align}
\kappa_{i+1}\left(\frac{q_{j-1}}{Q},\frac{q_j}{Q}\right)&=\kappa_1\circ
T^i\left(\frac{q_{j-1}}{Q},\frac{q_j}{Q}\right)\nonumber\\
&=\kappa_1\left(\frac{q_{j+i-1}}{Q},\frac{q_{j+i}}{Q}\right)\label{kform1}\\
&=\left[\frac{Q+q_{j+i-1}}{q_{j+i}}\right]\nonumber\\
&=\nu_2(\gamma_{j+i}).\nonumber
\end{align}
For ease of proof it will also be convenient to define for each positive integer $k$ a set
\begin{equation*}
\T_k^*=\bigcup_{\ell=k}^\infty \T_\ell.
\end{equation*}
It is easy to verify that
\begin{align}\label{arlenTk}
\ar \T_k^*=\frac{2}{k(k+1)}\quad\text{ and }\quad\len \T_k^*\ll \frac{1}{k}.
\end{align}
Furthermore from \cite[Corollary 2.5 and Remark 2.6]{boca2002} we have the following result.
\begin{lemma}\label{tkintlemma}
If $h,k,$ and $\ell$ are any positive integers with $\min (k,\ell)>c_h=4h+2$ then \[\T_k^*\cap T^{-h}\T_\ell^*=\emptyset.\]
\end{lemma}
Finally we point out (see also \cite[Remark 2.3]{boca2002}) that the matrices for $T$ as linear transformations on each of the sets $\T_k$ are elements of $SL_2(\Z)$. Setting
 \begin{equation*}
\Zv=\{(a,b)\in\Z^2:(a,b)=1\},
\end{equation*}
this implies that \[\#\{Q\Omega\cap\Zv\}=\#\{Q(T\Omega)\cap\Zv\}\]
for any $\Omega\subseteq\T.$

\begin{proof}[Proof of Theorem \ref{asympthm}]
First let us consider the case when $k=4$. By Theorem \ref{algidthm}
together with formula (\ref{K3form}) we have that
\begin{align}
\sum_{i=1}^{N(Q)}\nu_4(\gamma_i)&= \sum_{i=1}^{N(Q)}\left(\frac{7}{2}\right)K_3(-\nu_2(\gamma_i),\nu_2(\gamma_{i+1}),-\nu_2(\gamma_{i+2}))\nonumber\\
&=\sum_{i=1}^{N(Q)}\nu_2(\gamma_i)\nu_2(\gamma_{i+1})\nu_2(\gamma_{i+2})-\sum_{i=1}^{N(Q)}\nu_2(\gamma_i)-\sum_{i=1}^{N(Q)}\nu_2(\gamma_{i+2})\label{nu4sum1}\\
&=R_1(Q)-R_2(Q)-R_3(Q).\label{nu4sum2}
\end{align}
Appealing to (\ref{indsum1}) and to the periodicity of $\nu_2$ we have that
\begin{equation}
R_2(Q)=R_3(Q)=3N(Q)-1.
\end{equation}
To evaluate $R_1(Q)$ we will use the following well known fact about Farey fractions.
\begin{lemma}
Let $a,b,$ and $Q$ be positive integers. Then there is an integer $1\le i\le N(Q)$ for which $q_i=a$ and $q_{i+1}=b$ if and only if
\begin{equation*}
1\le a,b\le Q,\quad (a,b)=1,\quad\text{and}\quad a+b>Q.
\end{equation*}
Furthermore when these conditions on $a$ and $b$ are satisfied then the integer $i$ is uniquely determined.
\end{lemma}
Using (\ref{kform1}) together with this lemma we find that
\begin{align}
R_1(Q)&=\sum_{i=1}^{N(Q)}\kappa_1\left(\frac{q_{i-1}}{Q},\frac{q_i}{Q}\right)\kappa_2\left(\frac{q_{i-1}}{Q},\frac{q_i}{Q}\right)\kappa_3\left(\frac{q_{i-1}}{Q},\frac{q_i}{Q}\right)\nonumber\\
&=\sum_{(a,b)\in Q\T\cap\Zv}\kappa_1\left(\frac{a}{Q},\frac{b}{Q}\right)\kappa_2\left(\frac{a}{Q},\frac{b}{Q}\right)\kappa_3\left(\frac{a}{Q},\frac{b}{Q}\right).\label{kform2}
\end{align}
Now using (\ref{Tkform1}) and the fact that $Q\T\cap\Zv$ can be written as the disjoint union
\begin{equation*}
Q\T\cap\Zv=\bigcup_{k=1}^\infty(Q\T_k\cap\Zv)
\end{equation*}
we find that the right hand side of (\ref{kform2}) is equal to
\begin{align}
\sum_{k_1=1}^\infty\sum_{(a,b)\in Q\T_{k_1}\cap\Zv}&\kappa_1\left(\frac{a}{Q},\frac{b}{Q}\right)\kappa_2\left(\frac{a}{Q},\frac{b}{Q}\right)\kappa_3\left(\frac{a}{Q},\frac{b}{Q}\right)\nonumber\\
&=\sum_{k_1=1}^\infty k_1\sum_{(a,b)\in Q\T_{k_1}\cap\Zv}\kappa_2\left(\frac{a}{Q},\frac{b}{Q}\right)\kappa_3\left(\frac{a}{Q},\frac{b}{Q}\right)\nonumber\\
&=\sum_{k_1,k_2=1}^\infty k_1k_2\sum_{(a,b)\in Q(\T_{k_1}\cap T^{-1}\T_{k_2})\cap\Zv}\kappa_3\left(\frac{a}{Q},\frac{b}{Q}\right)\nonumber\\
&=\sum_{k_1,k_2,k_3=1}^\infty k_1k_2k_3\#\left\{Q(\T_{k_1}\cap T^{-1}\T_{k_2}\cap T^{-2}\T_{k_3})\cap\Zv\right\}.\nonumber
\end{align}
By interchanging the order of summation we find that the last sum is equal to
\begin{align}
&\sum_{\ell_1,\ell_2,\ell_3=1}^\infty\sum_{k_1=\ell_1}^\infty\sum_{k_2=\ell_2}^\infty\sum_{k_3=\ell_3}^\infty\#\left\{Q(\T_{k_1}\cap T^{-1}\T_{k_2}\cap T^{-2}\T_{k_3})\cap\Zv\right\}\nonumber\\
&\qquad\qquad\qquad=\sum_{\ell_1,\ell_2,\ell_3=1}^\infty\#\left\{Q(\T_{\ell_1}^*\cap T^{-1}\T_{\ell_2}^*\cap T^{-2}\T_{\ell_3}^*)\cap\Zv\right\}.\label{kform3}
\end{align}
Now we write
\begin{equation*}
A_{\ell_1,\ell_2,\ell_3}(Q)=\#\left\{Q(\T_{\ell_1}^*\cap T^{-1}\T_{\ell_2}^*\cap T^{-2}\T_{\ell_3}^*)\cap\Zv\right\}
\end{equation*}
and we split up the sum in (\ref{kform3}) as \[\sum_{i=0}^3S_i(Q)\] with
\begin{align*}
S_0(Q)&=\sum_{\ell_1,\ell_2,\ell_3=1}^{c_2}A_{\ell_1,\ell_2,\ell_3}(Q)\\
\intertext{and}
S_i(Q)&=\sum_{\ell_i=c_2+1}^{2Q}\sum_{\substack{j=1\\j\not= i}}^3\sum_{\ell_j=1}^{c_2}A_{\ell_1,\ell_2,\ell_3}(Q)~\text{ for } 1\le i\le 3.
\end{align*}
Note that here we are using Lemma \ref{tkintlemma} and the fact that $Q\T_k\cap\Zv=\emptyset$ for $k>2Q.$ Now we estimate each of these sums using the following classical result which we quote from \cite[Corollary 2.2]{boca2002}.
\begin{lemma}
If $\Omega\subseteq [0,R_1]\times[0,R_2]$ is a bounded region with rectifiable boundary and $R\ge \min (R_1,R_2)$ then
\[\#(\Omega\cap\Zv)=\frac{6\ar (\Omega)}{\pi^2}+O\left(R+\len (\partial\Omega)\log R+\frac{\ar (\Omega)}{R}\right).\]
\end{lemma}
By appealing to this lemma we find that
\begin{align*}
S_0(Q)&=\frac{6Q^2}{\pi^2}\sum_{\ell_1,\ell_2,\ell_3=1}^{c_2}\ar (\T_{\ell_1}^*\cap T^{-1}\T_{\ell_2}^*\cap T^{-2}\T_{\ell_3}^*)\\
&\qquad\qquad+O\left(Q\log Q\sum_{\ell_1,\ell_2,\ell_3=1}^{c_2}\len \left(\partial(\T_{\ell_1}^*\cap T^{-1}\T_{\ell_2}^*\cap T^{-2}\T_{\ell_3}^*)\right)\right)\\
&=\frac{6Q^2}{\pi^2}\sum_{\ell_1,\ell_2,\ell_3=1}^{c_2}\ar (\T_{\ell_1}^*\cap T^{-1}\T_{\ell_2}^*\cap T^{-2}\T_{\ell_3}^*)+O(Q\log Q)
\end{align*}
and that
\begin{align}
S_1(Q)&=\frac{6Q^2}{\pi^2}\sum_{\ell_1=c_2+1}^{2Q}\sum_{\ell_2,\ell_3=1}^{c_2}\ar (\T_{\ell_1}^*\cap T^{-1}\T_{\ell_2}^*\cap T^{-2}\T_{\ell_3}^*)\label{s1form1}\\
&\qquad+O\left(Q\log Q\sum_{\ell_1=c_2+1}^{2Q}\sum_{\ell_2,\ell_3=1}^{c_2}\len \left(\partial(\T_{\ell_1}^*\cap T^{-1}\T_{\ell_2}^*\cap T^{-2}\T_{\ell_3}^*)\right)\right).\nonumber
\end{align}
To estimate the error term in our formula for $S_1(Q)$ we use (\ref{arlenTk}) and observe that
\begin{align*}
&\len \left(\partial(\T_{\ell_1}^*\cap T^{-1}\T_{\ell_2}^*\cap T^{-2}\T_{\ell_3}^*)\right)\\
&\qquad\qquad\le\sum_{k_2=\ell_2}^{c_2}\sum_{k_3=\ell_3}^{c_2}\len \left(\partial(\T_{\ell_1}^*\cap T^{-1}\T_{k_2}\cap T^{-2}\T_{k_3})\right)\\
&\qquad\qquad\ll\frac{1}{\ell_1}.
\end{align*}
Furthermore extending the first sum on $\ell_1$ in (\ref{s1form1}) to infinity introduces an overall error of at most $O(Q)$. Thus we have that
\begin{align*}
S_1(Q)=\frac{6Q^2}{\pi^2}\sum_{\ell_1=c_2+1}^{\infty}\sum_{\ell_2,\ell_3=1}^{c_2}\ar (\T_{\ell_1}^*\cap T^{-1}\T_{\ell_2}^*\cap T^{-2}\T_{\ell_3}^*)+O\left(Q(\log Q)^2\right).
\end{align*}
Estimates for $S_2(Q)$ and $S_3(Q)$ are obtained in the same way and substituting everything back into (\ref{nu4sum2}) we have that
\begin{align*}
\sum_{i=1}^{N(Q)}\nu_4(\gamma_i)=B(4)N(Q)+O(Q(\log Q)^2)
\end{align*}
with
\begin{align*}
B(4)=2\sum_{\ell_1,\ell_2,\ell_3=1}^{\infty}\ar (\T_{\ell_1}^*\cap T^{-1}\T_{\ell_2}^*\cap T^{-2}\T_{\ell_3}^*)-6.
\end{align*}
\emph{It is important to note that the constant implied in the error term here depends at most on the quantity $c_2$.}

The proof for $k>4$ is virtually the same. We write the $k$th convergent polynomial as a sum of monomials
\begin{align*}
K_k(x_1,x_2,\ldots ,x_k)=\sum_{m=1}^{F_{k-1}}x_{j_{m,1}}x_{j_{m,2}}\cdots x_{j_{m,n_m}},
\end{align*}
where $F_{k-1}$ is the $(k-1)$st Fibonacci number and $j_{m,\ell}\in\{1,\ldots ,k\}$ for each $1\le m\le F_{k-1}$ and $1\le \ell\le n_m$. We also treat the empty product as being equal to $1$ and if $k$ is odd then we accommodate the constant term in $K_k$ by setting $n_{F_{k-1}}=0$. Then for each $m$ we define
\[R_m(Q)=\sum_{i=1}^{N(Q)}\nu_2(\gamma_{i+j_{m,1}-1})\nu_2(\gamma_{i+j_{m,2}-1})\cdots\nu_2(\gamma_{i+j_{m,n_m}-1}).\]
We may evaluate each of these sums as we did in the case when $k=4$, the only difference being that for each sum we would need to apply Lemma \ref{tkintlemma} with $h$ possibly as large as $k-1$. Thus in the end we obtain
\begin{align*}
\sum_{i=1}^{N(Q)}\nu_k(\gamma_i)=B(k)N(Q)+O_k(Q(\log Q)^2)
\end{align*}
with
\begin{align}\label{Bkval}
B(k)=2\sum_{m=1}^{F_{k-1}}\pm\sum_{i=1}^{n_m}\sum_{\ell_i=1}^{\infty}\ar (T^{1-j_{m,1}}\T_{\ell_1}^*\cap T^{1-j_{m,2}}\T_{\ell_2}^*\cap\cdots\cap T^{1-j_{m,n_m}}\T_{\ell_{n_m}}^*).
\end{align}
\end{proof}


\section{Closing remarks}
First of all we remark that for any $k\ge 2$ the index $\nu_k(\gamma_i)$ will take all positive integer values as $i$ and $Q$ vary. We leave this for the reader to verify. The frequency with which a given $k-$index takes a particular value can be determined by computing all $(k-1)-$tuples of positive integers which are solutions to (\ref{contiden}), and then computing the areas of the corresponding subregions of the Farey triangle. Of course many of these subregions will be empty, and it is not obvious whether or not there is a nice formula for the frequencies in general.

A related problem is that of determining good bounds for the constants $B(k)$ which appear in Theorem \ref{asympthm}. A trivial bound for $B(k)$ can be obtained by a straightforward application of Lemma \ref{tkintlemma} to equation (\ref{Bkval}). At the very least this gives
\[B(k)\ll F_{k-1}c_k^k,\]
but this is probably far from best possible. The investigation of this problem and of the best estimate for $A(h)$ seem to be related to the determination of the dynamical properties of the map $T$. Indeed a proof that $T$ is strongly mixing would likely also imply that $A(h)\ll 1$ as $h\rar\infty$. However the ergodic properties of this map are at this time unknown.



\end{document}